\documentclass[11pt]{amsart} 
\usepackage[foot]{amsaddr}
\usepackage{amsmath,amssymb,bm, color}

\usepackage{amsmath}
\usepackage{graphicx} 
\usepackage{float}
\usepackage{amsfonts,amssymb}
\usepackage{dsfont}
\usepackage{pifont}

\numberwithin{equation}{section}
 \newtheorem{theorem}{Theorem}[section]
 \newtheorem{lemma}[theorem]{Lemma}

\def\3bar{{|\hspace{-.02in}|\hspace{-.02in}|}}
\def\E{{\mathcal{E}}}
\def\T{{\mathcal{T}}}

\def\pT{{\partial T}}

\def\bu{{\mathbf{u}}}
\def\bv{{\mathbf{v}}}
\def\bQ{{\mathbf{Q}}}

\def\be{{\mathbf{e}}}

\newtheorem{algorithm}{Algorithm}[section]

 \def\ad#1{\begin{aligned}#1\end{aligned}}  \def\b#1{{\bf #1}} \def\hb#1{\hat{\bf #1}}
\def\a#1{\begin{align*}#1\end{align*}} 
\def\an#1{\begin{align}#1\end{align}}
\def\e#1{\begin{equation}#1\end{equation}} \def\t#1{\hbox{\rm{#1}}}
\def\dt#1{\left|\begin{matrix}#1\end{matrix}\right|}
\def\p#1{\begin{pmatrix}#1\end{pmatrix}} \def\c{\operatorname{curl}}
 \def\vc{\nabla\times } \numberwithin{equation}{section}
 \def\la{\circle*{0.25}}
\def\boxit#1{\vbox{\hrule height1pt \hbox{\vrule width1pt\kern1pt
     #1\kern1pt\vrule width1pt}\hrule height1pt }}
 \def\lab#1{\boxit{\small #1}\label{#1}}
  \def\mref#1{\boxit{\small #1}\ref{#1}}
 \def\meqref#1{\boxit{\small #1}\eqref{#1}}
\long\def\comment#1{}

\title [Least-Squares Weak Galerkin]{A Least-Squares Weak Galerkin Method for Second-Order Elliptic Equations in Non-Divergence Form}

  \author {Chunmei Wang$\dagger$}
  \address{Department of Mathematics, University of Florida, Gainesville, FL 32611, USA. }
  \email{chunmei.wang@ufl.edu}
 
\thanks{$\dagger$ \ Corresponding author. }

\author {Shangyou Zhang}
\address{Department of Mathematical Sciences,  University of Delaware, Newark, DE 19716, USA}   \email{szhang@udel.edu}

\begin{document}

\begin{abstract}This article proposes a novel least-squares weak Galerkin (LS-WG) method for second-order elliptic equations in non-divergence form. The approach leverages a locally defined discrete weak Hessian operator constructed within the weak Galerkin framework. A key feature of the resulting algorithm is that it yields a symmetric and positive definite linear system while remaining applicable to general polygonal and polyhedral meshes. We establish optimal-order error estimates for the approximation in a discrete $H^2$-equivalent norm. Finally, comprehensive numerical experiments are presented to validate the theoretical analysis and demonstrate the efficiency and robustness of the method.
\end{abstract}



\maketitle

 \def\ad#1{\begin{aligned}#1\end{aligned}}  \def\b#1{{\bf #1}} \def\hb#1{\hat{\bf #1}}
\def\a#1{\begin{align*}#1\end{align*}} 
\def\an#1{\begin{align}#1\end{align}}
\def\e#1{\begin{equation}#1\end{equation}} \def\t#1{\hbox{\rm{#1}}}
\def\dt#1{\left|\begin{matrix}#1\end{matrix}\right|}
\def\p#1{\begin{pmatrix}#1\end{pmatrix}} \def\c{\operatorname{curl}}
 \def\vc{\nabla\times } \numberwithin{equation}{section}
 \def\la{\circle*{0.25}}
\def\boxit#1{\vbox{\hrule height1pt \hbox{\vrule width1pt\kern1pt
     #1\kern1pt\vrule width1pt}\hrule height1pt }}
 \def\lab#1{\boxit{\small #1}\label{#1}}
  \def\mref#1{\boxit{\small #1}\ref{#1}}
 \def\meqref#1{\boxit{\small #1}\eqref{#1}}
\long\def\comment#1{}

\section{Introduction}
This work presents a novel numerical framework for second-order elliptic equations in non-divergence form. Specifically, we consider the model problem of finding an unknown function $u$ such that:
\begin{equation}\label{model_problem}
\begin{cases}
\sum_{i,j=1}^d a_{ij}\partial^2_{ij}u = f & \text{in } \Omega, \\
u = 0 & \text{on } \partial\Omega,
\end{cases}
\end{equation}
where $\Omega \subset \mathbb{R}^d$ ($d=2,3$) is an open bounded domain with a Lipschitz continuous boundary $\partial\Omega$. The coefficient tensor $a(x) = (a_{ij}(x))_{d \times d}$ is assumed to be symmetric and uniformly elliptic; that is, there exist positive constants $\alpha$ and $\beta$ such that
\begin{equation*}  
\alpha|\xi|^2 \leq \xi^T a(x)\xi \le \beta|\xi|^2, \quad \forall \xi \in \mathbb{R}^d, \, x \in \Omega.
\end{equation*} 
Under these assumptions, and provided $\Omega$ is convex and $a_{ij} \in L^\infty(\Omega)$, problem \eqref{model_problem} admits a unique strong solution $u \in H^2(\Omega) \cap H_0^1(\Omega)$ satisfying the a priori estimate $\|u\|_{2} \leq C \|f\|_0$ \cite{smears}.

Equations in non-divergence form are foundational in stochastic control and probability theory \cite{Fleming}. Furthermore, they frequently arise as linearized sub-problems when solving fully nonlinear partial differential equations (PDEs), such as the Hamilton--Jacobi--Bellman equations, via Newton-type iterations \cite{brenner-0, neilan}. A significant computational challenge in these applications is that the coefficient tensor $a(x)$ is often only essentially bounded or piecewise continuous. Consequently, traditional finite element methods that rely on high global regularity may exhibit suboptimal convergence or fail to converge entirely.
To address these challenges, various specialized numerical schemes for non-divergence form equations have been developed in recent years. These include conforming methods based on finite element Hessians \cite{lakkis, neilan}, primal methods employing interior penalties \cite{feng}, and $hp$-version discontinuous Galerkin (DG) methods of the least-squares type \cite{smears}. 

The Weak Galerkin (WG) finite element method has emerged as a versatile paradigm for the numerical solution of partial differential equations (PDEs). The core innovation of the WG framework lies in the reconstruction of discrete differential operators—such as gradient, divergence, and Hessian—through a distribution-like theory applied to piecewise polynomials. Unlike traditional finite element methods, the WG approach relaxes global regularity requirements by utilizing carefully designed stabilizers to enforce connection across element interfaces. The robustness and flexibility of the WG method have been extensively documented across a wide range of PDE models \cite{wg1, wg2, wg3, wg4, wg5, wg6, wg7, wg8, wg9, wg10, wg11, wg12, wg13, wg14, wg15, wg16, wg17, wg18, wg19, wg20, wg21, itera, wy3655}.

A defining characteristic of WG methods is their reliance on weak derivatives and discrete weak continuities to formulate numerical schemes. This structural flexibility allows the WG framework to handle general meshes, including those with hanging nodes or unconventional element geometries (e.g., polygonal and polyhedral meshes), while maintaining stability and optimal convergence rates.

A significant evolution within this framework is the \textit{Primal-Dual Weak Galerkin (PDWG)} method. This approach was specifically developed to address problems that pose challenges for standard variational or Galerkin-based techniques \cite{pdwg, pdwg1, pdwg2, pdwg3, pdwg4, pdwg5, pdwg6, pdwg7, pdwg8, pdwg9, pdwg10, pdwg11, pdwg12, pdwg13, pdwg14, pdwg15}. The PDWG method reformulates the underlying PDE as a constrained minimization problem, where the constraints are defined by the weak formulation of the differential equation using weak operators. The resulting Euler--Lagrange equations integrate both a primal variable and a dual variable (Lagrange multiplier), naturally yielding a symmetric linear system even for problems that are non-symmetric in their original form.

Although a PDWG  method was recently introduced  for the second order elliptic problem in non-divergence form \eqref{model_problem}  \cite{pdwg}, there remains a distinct need for a framework that inherently produces symmetric positive definite (SPD) systems while maintaining the flexibility to operate on general meshes.
The primary objective of this paper is to introduce a \textit{least-squares weak Galerkin (LS-WG)} method for problem \eqref{model_problem}. This approach offers several technical and computational advantages:
\begin{enumerate}
    \item \textbf{Robustness:} It provides a stable discretization for equations with rough (e.g., $L^\infty$ or piecewise continuous) coefficients.
    \item \textbf{Geometric Flexibility:} The method is applicable to general polygonal and polyhedral meshes, which is a hallmark of the weak Galerkin methodology.
    \item \textbf{Algorithmic Efficiency:} The resulting discrete linear system is symmetric and positive definite, enabling the use of highly efficient iterative solvers, such as the preconditioned conjugate gradient method.
\end{enumerate}

We establish optimal-order error estimates in a discrete $H^2$-equivalent norm. Furthermore, we provide a series of numerical experiments   to validate the theoretical analysis and demonstrate the method's practical performance.

The remainder of this paper is organized as follows. In Section 2, we review the construction of the discrete weak Hessian operator. Section 3 details the formulation of the LS-WG finite element method. In Section 4, we derive the optimal-order error estimates. Finally, numerical results are presented in Section 5 to illustrate the efficiency of the proposed scheme.
 
Throughout this paper, $D \subset \mathbb{R}^d$ denotes an open bounded domain with a Lipschitz continuous boundary. We employ standard notation for Sobolev spaces $H^s(D)$ with associated inner products $(\cdot, \cdot)_{s,D}$, norms $\|\cdot\|_{s,D}$, and seminorms $|\cdot|_{s,D}$ for $s \ge 0$ \cite{ciarlet-fem}. The $L^2$ inner product on the boundary $\partial D$ is denoted by $\langle \cdot, \cdot \rangle_{\partial D}$. For simplicity, subscripts are omitted when $D = \Omega$.
\section{Weak Hessian and Discrete Weak Hessian}

For classical functions, the Hessian is the square matrix of second-order partial derivatives. In the context of second-order elliptic problems in non-divergence form (\ref{model_problem}), the Hessian serves as the primary differential operator. Consequently, it is essential to develop numerical methods specifically tailored to the Hessian operator. This section reviews the discrete weak Hessian operator originally introduced in \cite{ww}.

Let $T$ be a polygonal or polyhedral domain with boundary $\partial T$. We define a  {weak function} on $T$ as a triplet $v=\{v_0, v_b, \mathbf{v}_g\}$, where $v_0 \in L^2(T)$, $v_b \in L^2(\partial T)$, and $\mathbf{v}_g \in [L^2(\partial T)]^d$. In this framework, the components $v_0$ and $v_b$ represent the values of $v$ in the interior and on the boundary of $T$, respectively. The third component, $\mathbf{v}_g = (v_{g1}, \dots, v_{gd})$, is intended to represent the gradient $\nabla v$ on $\partial T$. 

Note that $v_b$ and $\mathbf{v}_g$ are not required to be the traces of $v_0$ and $\nabla v_0$ on $\partial T$. If $v_0$ possesses sufficient regularity such that these traces exist and are used for $v_b$ and $\mathbf{v}_g$, the weak function $v$ is uniquely determined by $v_0$ and reduces to a classical function. One may also choose $v_b$ as the trace of $v_0$ while treating $\mathbf{v}_g$ as an independent variable, or vice versa. We denote the space of all weak functions on $T$ by:
\begin{equation*}\label{2.1}
W(T) = \{v = \{v_0, v_b, \mathbf{v}_g\} : v_0 \in L^2(T), v_b \in L^2(\partial T), \mathbf{v}_g \in [L^2(\partial T)]^d\}.
\end{equation*}

For any $v \in W(T)$, the  { weak second-order partial derivative}, denoted by $\partial^2_{ij,w} v$ ($i,j=1, \cdots, d$), is defined as a bounded linear functional on the Sobolev space $H^2(T)$. Its action on a test function $\varphi \in H^2(T)$ is defined by:
\begin{equation*}\label{2.3}
\langle \partial^2_{ij,w}v, \varphi \rangle_T := (v_0, \partial^2_{ji}\varphi)_T - \langle v_b n_i, \partial_j\varphi \rangle_{\partial T} + \langle v_{gi}, \varphi n_j \rangle_{\partial T},
\end{equation*}
where $\mathbf{n} = (n_1, \dots, n_d)$ denotes the unit outward normal vector on $\partial T$. The  {weak Hessian} of $v \in W(T)$ is then defined as the matrix:
\[
\nabla^2_{w,T} v = \left[ \partial_{ij,w}^2 v \right]_{d \times d}.
\]

To facilitate numerical implementation, we define a discrete version of the operator. Let $P_r(T)$ be the space of polynomials of total degree $r$ or less on $T$. The  {discrete weak second-order partial derivative}, denoted by $\partial^2_{ij,w,r,T}$, is the unique polynomial in $P_r(T)$ satisfying:
\begin{equation}\label{2.4}
(\partial^2_{ij,w,r,T} v, \varphi)_T = (v_0, \partial^2_{ji}\varphi)_T - \langle v_b n_i, \partial_j\varphi \rangle_{\partial T} + \langle v_{gi}, \varphi n_j \rangle_{\partial T},
\end{equation}
for all $\varphi \in P_r(T)$.
The {discrete weak Hessian} $\nabla^2_{w,r,T} v$ is defined analogously as the matrix of these discrete partial derivatives:
\[
\nabla^2_{w,r,T} v = \left[ \partial_{ij,w,r,T}^2 v \right]_{d \times d}.
\]

If $v \in W(T)$ has a smooth interior component $v_0 \in H^2(T)$, applying standard integration by parts to the first term on the right-hand side of (\ref{2.4}) yields the following identity for all $\varphi \in P_r(T)$:
\begin{equation}\label{2.4new}
(\partial^2_{ij,w,r,T} v, \varphi)_T = (\partial^2_{ij} v_0, \varphi)_T - \langle (v_b - v_0) n_i, \partial_j \varphi \rangle_{\partial T} + \langle v_{gi} - \partial_i v_0, \varphi n_j \rangle_{\partial T}.
\end{equation}
 
\section{Least Squares Weak Galerkin}\label{Section:PD-WGFEM}
Let ${\mathcal  T}_h$ be a finite element partition of the domain
$\Omega$ into polygons in 2D or polyhedra in 3D. Denote by
${\mathcal E}_h$ the set of all edges or flat faces in ${\mathcal  T}_h$
and ${\mathcal E}_h^0={\mathcal E}_h \setminus \partial\Omega$ the
set of all interior edges or flat faces. Assume that ${\mathcal  T}_h$
satisfies the shape regularity conditions described as in
\cite{wy3655}. Denote by $h_T$ the diameter of $T\in {\mathcal  T}_h$ and
$h=\max_{T\in {\mathcal  T}_h}h_T$ the meshsize of the partition ${\mathcal 
T}_h$.  

For any given integer $k\geq 2$, let $W_k(T)\subset W(T)$ be a
subspace consisting of   polynomials in the following form
\begin{equation*}\label{EQ:local-weak-fem-space}
W_k(T):=\{v=\{v_0,v_b,\bv_g\}\in P_k(T)\times P_k(e)\times
[P_{k-1}(e)]^d,\ e\in \partial T\cap\E_h\}.
\end{equation*}
By patching $W_k(T)$ over all $T\in {\mathcal  T}_h$ through a common
value on the interior interface $\E_h^0$ for $v_b$, we arrive at
the following weak finite element space
$$
W_{h}:=\big\{\{v_0,v_b, \textbf{v}_g\}:\ \{v_0,v_b, \bv_g\}|_T\in
W_k(T), \ T\in {\mathcal  T}_h\big\}.
$$
Denote by $W_{h}^0$ the subspace of $W_{h}$ with vanishing
boundary value for $v_b$ on $\partial\Omega$:
\begin{equation*} 
W_{h}^0=\{\{v_0,v_b, \bv_g\}\in W_{h},\ v_b|_e=0, e\subset
\partial\Omega\}.
\end{equation*}

For simplicity of notation, we denote by $\partial^2_{ij, w}$ the
discrete weak second order partial differential operator defined by
(\ref{2.4}); i.e.,
$$
(\partial^2_{ij, w} v)|_T=\partial^2_{ij,w,r,T}(v|_T), \qquad \forall v\in
W_{h}.
$$
 We define the bilinear form
 $$
 a(u, v)=\sum_{T\in {\mathcal  T}_h}   (\sum_{i,j=1}^d a_{ij}\partial_{ij,w}^2
u, \sum_{i,j=1}^d a_{ij}\partial_{ij,w}^2
v)_T,
 $$
 and the stabilizer 
 $$s(u, v)=\sum_{T\in {\mathcal  T}_h} h_T^{-3} \langle u_0-u_b,
v_0-v_b\rangle_\pT+ h_T^{-1}  \langle \nabla u_0 -\bu_g, \nabla v_0 -\bv_g \rangle_\pT. $$

\begin{algorithm}\emph{(Least Squares Weak Galerkin FEM)}
\label{ALG:primal-dual-wg-fem} For a numerical approximation of the
second order elliptic problem    in the non-divergence form (\ref{model_problem}),
find $ u_h \in W_{h}^0$ satisfying
 \begin{equation}\label{fem}
a(u_h, v)  +s(u_h, v)= \sum_{T\in {\mathcal  T}_h} (f, \sum_{i,j=1}^d  a_{ij}\partial_{ij,w}^2
v)_T,  
\end{equation}
for any $v\in W_{h}^0$.
 
\end{algorithm}

\begin{theorem}\label{theorem1}  
 Assume that the problem (\ref{model_problem}) has a
unique strong solution. The least-square weak Galerkin scheme \eqref{fem} has a unique solution.
\end{theorem} 
\begin{proof}
 It suffices to show that the solution of \eqref{fem} is trivial if $f=0$. Assume that $f=0$ and let $v=u_h$ in \eqref{fem}, we have 
 $$ 
\sum_{T\in {\mathcal  T}_h}   (\sum_{i,j=1}^d a_{ij}\partial_{ij,w}^2
u_h, \sum_{i,j=1}^d a_{ij}\partial_{ij,w}^2
u_h)_T +s(u_h, u_h)=0,   
 $$
 which gives $\sum_{i,j=1}^d a_{ij}\partial_{ij,w}^2
u_h=0$ on each element $T$, $u_0=u_b$   and $\nabla u_0=\bu_g$ on each $\pT$.

 Using $u_0=u_b$  and $\nabla u_0=\bu_g$ on each $\pT$ gives that $u_0$ and $\nabla u_0$ is continuous in the domain $\Omega$ and thus $u_0\in C^1(\Omega)$.

It follows from  \eqref{2.4new}, $u_0=u_b$ and $\nabla u_0=\bu_g$ on each $\pT$ that
 \begin{equation*} 
 (\partial^2_{ij,w}u,\varphi)_T=(\partial^2
 _{ij}u_0,\varphi)_T,
 \end{equation*}
 for all $\varphi \in P_{k-2}(T)$. This leads to $ \partial^2_{ij,w}u=\partial^2
 _{ij}u_0 $  on each element $T$. Thus, the fact $\sum_{i,j=1}^d a_{ij}\partial_{ij,w}^2
u_h=0$  on each element $T$ implies $\sum_{i,j=1}^d a_{ij}\partial_{ij}^2 
u_0=0$  on each element $T$. Recall that $u_0\in C^1(\Omega)$, we   further have  $\sum_{i,j=1}^d a_{ij}\partial_{ij}^2 
u_0=0$  in $\Omega$.  It follows from $u_0=u_b$ on each $\pT$ and $u_b=0$ on $\partial \Omega$ gives $u_0=0$ on $\partial \Omega$. From $\sum_{i,j=1}^d a_{ij}\partial_{ij}^2 
u_0=0$ in $\Omega$ and $u_0=0$ on $\partial \Omega$, we conclude that $u_0 \equiv 0$ in $\Omega$ based on we assume that the problem (\ref{model_problem}) has a
unique strong solution. Using $u_0=u_b$ on each $\pT$ and $\nabla u_0=\bu_g$ on each $\pT$ gives $u_b\equiv 0$ and $\bu_g \equiv 0$ in $\Omega$. This gives $u_h \equiv 0$ in the domain $\Omega$.

This completes the proof of the theorem. 
\end{proof} 

We now define a semi-norm on $W_h$ as follows:
$$
\3bar v\3bar=(a(v,v)+s(v,v))^{\frac{1}{2}}.
$$
Similar to the proof of Theorem \ref{theorem1}, we can prove that $\3bar \cdot\3bar$ defines a norm on $W_h^0$.

\section{Error Estimates}
We begin by recalling the fundamental trace inequalities necessary for our analysis. Let ${\mathcal  T}_h$ be a shape-regular finite element partition of the domain $\Omega$. For any element $T \in {\mathcal  T}_h$ and function $\phi \in H^1(T)$, the following trace inequality holds \cite{wy3655}:
\begin{equation}\label{tracein}
\|\phi\|^2_{\partial T} \leq C \left( h_T^{-1} \|\phi\|_T^2 + h_T \|\nabla \phi\|_T^2 \right).
\end{equation}
When $\phi$ is a polynomial, this bound simplifies to the discrete trace inequality \cite{wy3655}:
\begin{equation}\label{trace}
\|\phi\|^2_{\partial T} \leq C h_T^{-1} \|\phi\|_T^2.
\end{equation}

For each element $T$, let $Q_0$ denote the $L^2$ projection onto the polynomial space $P_k(T)$ with $k \geq 2$. For each edge or face $e \subset \partial T$, let $Q_b$ and $\mathbf{Q}_g = (Q_{g1}, Q_{g2}, \dots, Q_{gd})$ represent the $L^2$ projections onto $P_k(e)$ and $[P_{k-1}(e)]^d$, respectively. 

For any function $w \in H^2(\Omega)$, we define $Q_h w$ as the $L^2$ projection onto the weak finite element space $W_k(T)$ such that on each element $T$:
\begin{equation*}
Q_h w = \{Q_0 w, Q_b w, \mathbf{Q}_g(\nabla w)\}.
\end{equation*}
Furthermore, let ${\mathcal  Q}_h$ denote the $L^2$ projection onto the space $P_{k-2}(T)$.

\begin{lemma}\label{Lemma5.2}\cite{wy3655}  Let ${\mathcal  T}_h$ be a
finite element partition of $\Omega$ satisfying the shape regularity
assumptions given in \cite{wy3655}. Then, for any $0\leq s\leq 1$
and $1\leq m\leq k$, one has
\begin{eqnarray}\label{3.1}
\sum_{T\in {\mathcal  T}_h}h_T^{2s}\|u-Q_0u\|^2_{s,T} &\leq &
Ch^{2(m+1)}\|u\|_{m+1}^2,\\
\label{3.3} \sum_{T\in {\mathcal 
T}_h}\sum_{i,j=1}^dh_T^{2s}\|\partial^2_{ij}u-{\mathcal 
Q}_h\partial^2_{ij}u\|^2_{s,T} &\leq& Ch^{2(m-1)}\|u\|_{m+1}^2.
\end{eqnarray}
\end{lemma}

\begin{lemma}\label{Lemma5.1} \cite{pdwg} The projection operators $Q_h$ and
${\mathcal  Q}_h$ satisfy the following commutative property:
\begin{equation}\label{EQ:CommutativeP}
\partial^2_{ij,w}(Q_h w)={\mathcal  Q}_h(\partial^2_{ij} w),\qquad
i,j=1,\ldots,d,
\end{equation}
for all $w\in H^2(T)$.
\end{lemma}

\begin{theorem}
   Let $u$ be the exact solution to the second-order elliptic problem in non-divergence form \eqref{model_problem}, and let $u_h \in W_h^0$ be the numerical solution to the least-squares Weak Galerkin scheme \eqref{fem}. We define the error function $e_h$ as  
   $$e_h=Q_hu-u_h.$$ 
    There exists a constant $C$ such that 
    \begin{equation}\label{est1}
        \3bar Q_hu-u_h\3bar \leq Ch^{k-1} \|u\|_{k+1}.
    \end{equation}
\end{theorem}
\begin{proof}
    Testing the first equation in \eqref{model_problem} by $ \sum_{i,j=1}^d a_{ij}\partial_{ij,w}^2 v$ implies
    $$
    \sum_{T\in {\mathcal  T}_h}(\sum_{i,j=1}^d a_{ij}\partial_{ij}^2 
u, \sum_{i,j=1}^d a_{ij}\partial_{ij,w}^2 v)_T=\sum_{T\in {\mathcal  T}_h}(f, \sum_{i,j=1}^d a_{ij}\partial_{ij,w}^2 v)_T.
    $$

    Using \eqref{EQ:CommutativeP},  we have
    \begin{equation*}\label{s1}
    \begin{split}
  &\sum_{T\in {\mathcal  T}_h} (\sum_{i,j=1}^d a_{ij}\partial_{ij}^2 
u, \sum_{i,j=1}^d a_{ij}\partial_{ij,w}^2 v)_T\\ =&\sum_{T\in {\mathcal  T}_h} (\sum_{i,j=1}^d  a_{ij}^2  \partial_{ij}^2 
u, \sum_{i,j=1}^d  \partial_{ij,w}^2 v)_T  \\
=& \sum_{T\in {\mathcal  T}_h}(\sum_{i,j=1}^d  {\mathcal  Q}_h (a_{ij}^2  \partial_{ij}^2 
u), \sum_{i,j=1}^d  \partial_{ij,w}^2 v)_T \\
 =& \sum_{T\in {\mathcal  T}_h}(\sum_{i,j=1}^d  {\mathcal  Q}_h (a_{ij}^2-\overline{a _{ij}^2})  \partial_{ij}^2 
u, \sum_{i,j=1}^d  \partial_{ij,w}^2 v)_T+(\sum_{i,j=1}^d {\mathcal  Q}_h   \overline{a _{ij}^2}   \partial_{ij}^2 
u, \sum_{i,j=1}^d  \partial_{ij,w}^2 v)_T \\
 =&\sum_{T\in {\mathcal  T}_h} (\sum_{i,j=1}^d  {\mathcal  Q}_h (a_{ij}^2-\overline{a _{ij}^2})  \partial_{ij}^2 
u, \sum_{i,j=1}^d  \partial_{ij,w}^2 v)_T+(\sum_{i,j=1}^d     \overline{a _{ij}^2}   \partial_{ij,w }^2 
Q_hu, \sum_{i,j=1}^d  \partial_{ij,w}^2 v)_T\\
=&\sum_{T\in {\mathcal  T}_h} (\sum_{i,j=1}^d  {\mathcal  Q}_h (a_{ij}^2-\overline{a _{ij}^2})  \partial_{ij}^2 
u, \sum_{i,j=1}^d  \partial_{ij,w}^2 v)_T+(\sum_{i,j=1}^d      (\overline{a _{ij}^2}- a_{ij}^2)  \partial_{ij,w }^2 
Q_hu, \sum_{i,j=1}^d  \partial_{ij,w}^2 v)_T\\&+(\sum_{i,j=1}^d  a_{ij}^2   \partial_{ij,w }^2 
Q_hu, \sum_{i,j=1}^d  \partial_{ij,w}^2 v)_T\\
=&\sum_{T\in {\mathcal  T}_h}(f, \sum_{i,j=1}^d a_{ij}\partial_{ij,w}^2 v)_T,
\end{split}
\end{equation*}
where $\overline{a _{ij}^2}$ denotes the average of $a _{ij}^2$ on each element, i.e., $\overline{a _{ij}^2}=\frac{1}{|T|}
\int_T a _{ij}^2 dT$. 

This gives
\begin{equation}\label{s3}
    \begin{split}
&\sum_{T\in {\mathcal  T}_h}( \sum_{i,j=1}^da_{ij}^2   \partial_{ij,w }^2 
Q_hu, \sum_{i,j=1}^d  \partial_{ij,w}^2 v)_T
+s(Q_hu, v)\\
=&\sum_{T\in {\mathcal  T}_h}(f, \sum_{i,j=1}^d a_{ij}\partial_{ij,w}^2 v)_T-   (\sum_{i,j=1}^d   {\mathcal  Q}_h (a_{ij}^2-\overline{a _{ij}^2})  \partial_{ij}^2 
u, \sum_{i,j=1}^d  \partial_{ij,w}^2 v)_T\\
&- (\sum_{i,j=1}^d     (\overline{a _{ij}^2}- a_{ij}^2)  \partial_{ij,w }^2 
Q_hu, \sum_{i,j=1}^d  \partial_{ij,w}^2 v)_T
+s(Q_hu, v). \end{split}
\end{equation}   

Subtracting \eqref{s3} from \eqref{fem} gives
\begin{equation}\label{ss}
    \begin{split}
&\sum_{T\in {\mathcal  T}_h}(\sum_{i,j=1}^d a_{ij}^2   \partial_{ij,w }^2 
(Q_hu-u_h), \sum_{i,j=1}^d  \partial_{ij,w}^2 v)_T
+s(Q_hu-u_h, v)\\
=&-\sum_{T\in {\mathcal  T}_h} ( \sum_{i,j=1}^d   {\mathcal  Q}_h (a_{ij}^2-\overline{a _{ij}^2})  \partial_{ij}^2 
u, \sum_{i,j=1}^d  \partial_{ij,w}^2 v)_T\\
&- (\sum_{i,j=1}^d      (\overline{a _{ij}^2}- a_{ij}^2)  \partial_{ij,w }^2 
Q_hu, \sum_{i,j=1}^d  \partial_{ij,w}^2 v)_T
+s(Q_hu, v).
 \end{split}
\end{equation}  
Letting $v=Q_hu-u_h$ in \eqref{ss} and using \eqref{EQ:CommutativeP} yields
\begin{equation}\label{sss}
    \begin{split}
&\3bar Q_hu-u_h\3bar^2\\
=&-\sum_{T\in {\mathcal  T}_h}  (\sum_{i,j=1}^d  {\mathcal  Q}_h ( a_{ij}^2-\overline{a _{ij}^2} ) \partial_{ij}^2 
u, \sum_{i,j=1}^d  \partial_{ij,w}^2 (Q_hu-u_h))_T\\
&-\sum_{T\in {\mathcal  T}_h}\sum_{i,j=1}^d (    (\overline{a _{ij}^2}- a_{ij}^2)  {\mathcal  Q}_h(\partial_{ij }^2 
 u), \sum_{i,j=1}^d  \partial_{ij,w}^2 (Q_hu-u_h))_T
\\&+s(Q_hu, Q_hu-u_h)\\
 =&-\sum_{T\in {\mathcal  T}_h}  (\sum_{i,j=1}^d  {\mathcal  Q}_h   a_{ij}^2  \partial_{ij}^2 
u, \sum_{i,j=1}^d  \partial_{ij,w}^2 (Q_hu-u_h))_T\\
&-\sum_{T\in {\mathcal  T}_h}\sum_{i,j=1}^d (     \ - a_{ij}^2   {\mathcal  Q}_h(\partial_{ij }^2 
 u), \sum_{i,j=1}^d  \partial_{ij,w}^2 (Q_hu-u_h))_T
\\&+s(Q_hu, Q_hu-u_h)\\
=&-\sum_{T\in {\mathcal  T}_h}  (\sum_{i,j=1}^d    a_{ij}  \partial_{ij}^2 
u, \sum_{i,j=1}^d  a_{ij}  \partial_{ij,w}^2 (Q_hu-u_h))_T\\
&+\sum_{T\in {\mathcal  T}_h}\sum_{i,j=1}^d (    a_{ij}{\mathcal  Q}_h(\partial_{ij }^2 
 u), \sum_{i,j=1}^d  a_{ij}\partial_{ij,w}^2 (Q_hu-u_h))_T
\\&+s(Q_hu, Q_hu-u_h)\\
=& \sum_{T\in {\mathcal  T}_h}  (\sum_{i,j=1}^d  a_{ij} ({\mathcal  Q}_h-I)    \partial_{ij}^2 
u, \sum_{i,j=1}^d  a_{ij}  \partial_{ij,w}^2 (Q_hu-u_h))_T 
\\&+s(Q_hu, Q_hu-u_h)\\
=& \sum_{T\in {\mathcal  T}_h}  (\sum_{i,j=1}^d  (a^2_{ij}-\overline{a^2_{ij}})   ({\mathcal  Q}_h-I) \partial_{ij}^2 
u, \sum_{i,j=1}^d     \partial_{ij,w}^2 (Q_hu-u_h))_T  
\\&+s(Q_hu, Q_hu-u_h).
\end{split}
\end{equation}  
Note that $a_{ij}$ is uniformly piecewise continuous in $\Omega$
with respect to $\T_h$. Thus, for any $\varepsilon >0$, there exists
a $h_0>0$ such that $\|a^2_{ij} - \overline{a^2_{ij}}\|_{L^\infty} \leq
\varepsilon$.  
Using Cauchy-Schwarz inequality,  the estimate \eqref{3.3} with $m=1$ and $s=0$, we have
\begin{equation}\label{ee1}
    \begin{split}
    & \Big|\sum_{T\in {\mathcal  T}_h}   (\sum_{i,j=1}^d  (a^2_{ij}-\overline{a^2_{ij}})  ({\mathcal  Q}_h-I)   \partial_{ij}^2 
u, \sum_{i,j=1}^d     \partial_{ij,w}^2 (Q_hu-u_h))_T\Big| \\
\leq & (\sum_{T\in {\mathcal  T}_h}\|\sum_{i,j=1}^d   (a^2_{ij}-\overline{a^2_{ij}}) ({\mathcal  Q}_h-I)  \partial_{ij}^2 
u\|^2_T)^{\frac{1}{2}}\\& \cdot(\sum_{T\in {\mathcal  T}_h} \|\sum_{i,j=1}^d     \partial_{ij,w}^2 (Q_hu-u_h)\|^2_T)^{\frac{1}{2}}  \\
\leq &C \epsilon \|u\|_2 \3bar Q_hu-u_h\3bar.
 \end{split}
\end{equation}

Recall 
$e_h=\{e_0, e_b, \be_g\}=Q_hu-u_h.$
Using Cauchy-Schwarz inequality, and the trace inequality   \eqref{tracein}, and the estimate \eqref{3.1} with $s=0,1$ and $m=k$, we obtain
\begin{equation}\label{ee2}
\begin{split}
&|s(Q_hu,e_h)| \\=& |\sum_{T\in {\mathcal  T}_h} h_T^{-3} \langle Q_0u-Q_bu,
e_0-e_b\rangle_\pT\\&+ h_T^{-1} \langle \nabla Q_0u -\bQ_g(\nabla u), \nabla e_0 -\be_g \rangle_\pT |\\
\leq &  (\sum_{T\in {\mathcal  T}_h}h_T^{-3}\|Q_0u-Q_bu\|^2_{\partial
T} )^{\frac{1}{2}} (\sum_{T\in {\mathcal  T}_h} h_T^{-3}
\|e_0-e_b\|^2_{\partial
T} )^{\frac{1}{2}}\\
&+   (\sum_{T\in {\mathcal  T}_h} h_T^{-1} \|\nabla Q_0u -\bQ_g(\nabla u)\|^2_{\partial T}\Big)^{\frac{1}{2}} (\sum_{T\in {\mathcal  T}_h}
h_T^{-1} \|\nabla e_0 -\be_g\|^2_{\partial
T} )^{\frac{1}{2}}\\
 \leq&   (\sum_{T\in {\mathcal  T}_h}h_T^{-4}\|Q_0u- u\|^2_{
T}+h_T^{-2}\|Q_0u- u\|^2_{
1, T} )^{\frac{1}{2}}\3bar e_h\3bar \\&+  (\sum_{T\in {\mathcal  T}_h} h_T^{-2} \|\nabla Q_0u -  \nabla u \|^2  +\|\nabla Q_0u -  \nabla u \|^2_{1, T} )^{\frac{1}{2}}  \3bar e_h\3bar\\
\leq & Ch^{k-1}\|u\|_{k+1}\3bar e_h\3bar.
\end{split}
\end{equation}

Substituting \eqref{ee1} and \eqref{ee2} into \eqref{sss} gives
$$
\3bar Q_hu-u_h\3bar\leq C h^{k-1}\|u\|_{k+1}.
$$

This completes the proof of the theorem.
\end{proof}

\begin{theorem}
Let $u$ and $u_h$ be the solution of model problem \eqref{model_problem} and the least square WG scheme \eqref{fem}, respectively. There exisits a constant $C$ such that 
\begin{equation}
    \3bar u-u_h\3bar \leq Ch^{k-1}\|u\|_{k+1}.
\end{equation}
\end{theorem}
\begin{proof}
By triangle inequality, Cauchy-Schwarz inequality,  \eqref{EQ:CommutativeP}, the error estimate \eqref{est1}, the estimate \eqref{3.3} with $s=0$ and $m=k$,   we have 
    \begin{equation}\label{q}
        \begin{split}
          &  \3bar u-u_h\3bar^2 \\ \leq   &  \3bar u-Q_hu\3bar^2+ \3bar  Q_hu-u_h\3bar^2 \\
            \leq & \sum_{T\in {\mathcal  T}_h}      \| \sum_{i,j=1}^da_{ij}\partial_{ij,w}^2
(u-Q_hu) \|^2_T +s(u-Q_hu, u-Q_hu)\\&+\3bar  Q_hu-u_h\3bar^2\\
\leq & \sum_{T\in {\mathcal  T}_h}      \| \sum_{i,j=1}^d a_{ij}(\partial_{ij}^2
 u-{\mathcal  Q}_h \partial_{ij}^2 u) \|^2_T +s(u-Q_hu, u-Q_hu)\\&+ \3bar  Q_hu-u_h\3bar^2\\
 \leq & Ch^{2k-2}\|u\|_{k+1}^2+s(u-Q_hu, u-Q_hu).
        \end{split}
    \end{equation}

  Using  the trace inequality \eqref{tracein}, the estimate \eqref{3.1} with $s=0, 1$ and $m=k$, we have
\begin{equation}\label{111}
\begin{split}
& s(u-Q_hu,u-Q_hu)  \\=& \sum_{T\in {\mathcal  T}_h} h_T^{-3} \| Q_bu-Q_0u 
\|^2_\pT+ h_T^{-1} \|   \bQ_g(\nabla u)-\nabla Q_0u\|^2_\pT  \\
 \leq&    \sum_{T\in {\mathcal  T}_h}h_T^{-4}\|Q_0u- u\|^2_{
T}+h_T^{-2}\|Q_0u- u\|^2_{
1, T}   \\&+   h_T^{-2} \|\nabla Q_0u -  \nabla u \|^2  +\|\nabla Q_0u -  \nabla u \|^2_{1, T}  \\
\leq & Ch^{2(k-1)}\|u\|^2_{k+1}.
\end{split}
\end{equation}

Substituting \eqref{111} into \eqref{q} completes the proof   of the theorem. 
\end{proof}

\section{Numerical Experiments}

In the first numerical test,  we solve the elliptic equation in non-divergence form \eqref{model_problem}
   on the unit square domain $\Omega=(0,1)\times(0,1)$, where 
\a{ a_{ij} = 1+\delta_{ij}, \quad i,j=1,2. } 
We choose an $f$ in  \eqref{model_problem} so that the exact solution is
\an{\label{s2}
   u = (x-x^3) (y^2-y^3) .  }  
 We compute the solution \eqref{s2} on the triangular grids shown in Figure \ref{f-2}, and on the non-convex polygonal 
  grids shown in Figure \ref{f-5}, by 
  the weak Galerkin $P_k$-$P_k$-$P_{k-1}^2$/$P_k^{2\times 2}$ finite elements, $k=2,3,4$ and $5$.
The results are listed in Tables \ref{t1}-\ref{t4}, 
    where we can see that the optimal orders of convergence 
  are achieved in all cases.  
  In these tables, $G_i$ denotes the $i$-th grid.  For example, $G_1$ in Table \ref{t1} is shown in
    Figure \ref{f-2} or Figure \ref{f-5} .
    
\begin{figure}[H]
\begin{center}\setlength\unitlength{2.4pt}\centering 
 \begin{picture}(140,45)(0,0) \put(0,41){$G_1:$}  \put(50,41){$G_2:$} \put(100,41){$G_3:$} 
  
\def\sq{\begin{picture}(40,40)(0,0) \put(0,40){\line(1,-1){40}}
  \multiput(0,0)(40,0){2}{\line(0,1){40}}\multiput(0,0)(0,40){2}{\line(1,0){40}} \end{picture} }
  
\put(0,0){\begin{picture}(40,40)(0,0)
  \multiput(0,0)(0,40){1}{\multiput(0,0)(40,0){1}{\sq}} 
  \end{picture} }
  
\put(50,0){\setlength\unitlength{1.2pt}\begin{picture}(40,40)(0,0)
  \multiput(0,0)(0,40){2}{\multiput(0,0)(40,0){2}{\sq}} 
  \end{picture} } 
\put(100,0){\setlength\unitlength{0.6pt}\begin{picture}(40,40)(0,0)
  \multiput(0,0)(0,40){4}{\multiput(0,0)(40,0){4}{\sq}} 
  \end{picture} } 
\end{picture}\end{center}
\caption{The triangular  grids for computing  \eqref{s2} in Tables \ref{t1}--\ref{t4}. }
\label{f-2}
\end{figure}
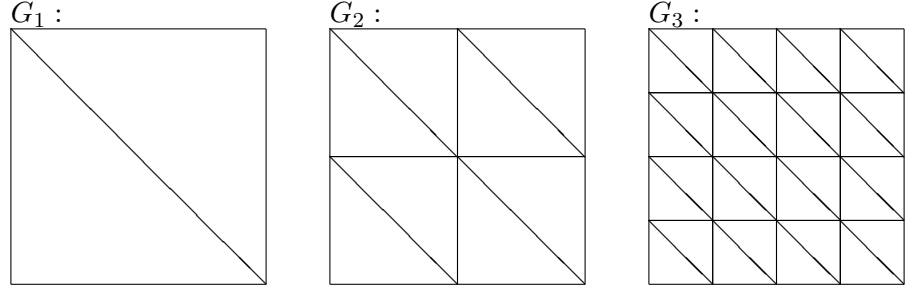

\begin{figure}[H]
\begin{center}\setlength\unitlength{2.4pt}\centering 
 \begin{picture}(140,45)(0,0) \put(0,41){$G_1:$}  \put(50,41){$G_2:$} \put(100,41){$G_3:$} 
  
\def\sq{\begin{picture}(40,40)(0,0) \put(0,0){\line(2,3){13.33}}  \put(40,40){\line(-2,-3){13.33}} \put(13.33,20){\line(1,0){13.33}}
  \multiput(0,0)(40,0){2}{\line(0,1){40}}\multiput(0,0)(0,40){2}{\line(1,0){40}} \end{picture} }
  
\put(0,0){\begin{picture}(40,40)(0,0)
  \multiput(0,0)(0,40){1}{\multiput(0,0)(40,0){1}{\sq}} 
  \end{picture} }
  
\put(50,0){\setlength\unitlength{1.2pt}\begin{picture}(40,40)(0,0)
  \multiput(0,0)(0,40){2}{\multiput(0,0)(40,0){2}{\sq}} 
  \end{picture} } 

\put(100,0){\setlength\unitlength{0.6pt}\begin{picture}(40,40)(0,0)
  \multiput(0,0)(0,40){4}{\multiput(0,0)(40,0){4}{\sq}} 
  \end{picture} } 
\end{picture}\end{center}
\caption{The non-convex polygonal grids for computing  \eqref{s2} in Tables \ref{t1}--\ref{t4}. }
\label{f-5}
\end{figure}
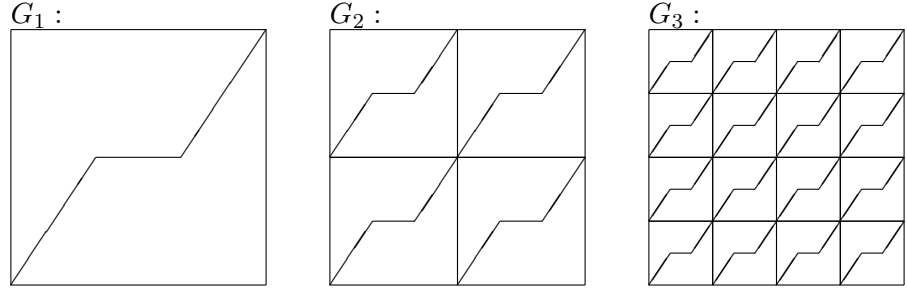

\begin{table}[H]
  \centering  \renewcommand{\arraystretch}{1.1}
  \caption{Error profile by the $P_2$ WG element  for computing \eqref{s2}. }
  \label{t1}
\begin{tabular}{c|cc|cc}
\hline
Grid $G_i$ & \quad $\| u-u_h\|_{0}$ & $O(h^r)$ & \  $\|D^2_w(u-u_h)\|_{0}$& $O(h^r)$   \\ \hline
    &  \multicolumn{4}{c}{ On the triangular meshes shown in Figure \ref{f-2}. }   \\
\hline  
 2&    0.287E-02 &  2.3&    0.831E+00 &  0.1\\
 3&    0.330E-03 &  3.1&    0.355E+00 &  1.2\\
 4&    0.686E-04 &  2.3&    0.184E+00 &  0.9\\
 5&    0.313E-04 &  1.1&    0.927E-01 &  1.0\\
\hline 
    &  \multicolumn{4}{c}{ On the  polygonal meshes shown in Figure \ref{f-5}. }   \\
\hline   
 2&    0.108E-01 &  2.9&    0.172E+01 &  0.7\\
 3&    0.155E-02 &  2.8&    0.808E+00 &  1.1\\
 4&    0.201E-03 &  2.9&    0.424E+00 &  0.9\\
 5&    0.291E-04 &  2.8&    0.216E+00 &  1.0\\
\hline 
    \end{tabular}%
\end{table}%

\begin{table}[H]
  \centering  \renewcommand{\arraystretch}{1.1}
  \caption{Error profile by the $P_3$ WG element  for computing \eqref{s2}. }
  \label{t2}
\begin{tabular}{c|cc|cc}
\hline
Grid $G_i$ & \quad $\| u-u_h\|_{0}$ & $O(h^r)$ & \  $\|D^2_w(u-u_h)\|_{0}$& $O(h^r)$   \\ \hline
    &  \multicolumn{4}{c}{ On the triangular meshes shown in Figure \ref{f-2}. }   \\ 
\hline  
 2&    0.477E-03 &  3.8&    0.306E+00 &  1.8\\
 3&    0.305E-04 &  4.0&    0.768E-01 &  2.0\\
 4&    0.192E-05 &  4.0&    0.181E-01 &  2.1\\
 5&    0.125E-06 &  3.9&    0.436E-02 &  2.1\\
\hline 
    &  \multicolumn{4}{c}{ On the polygonal meshes shown in Figure \ref{f-5}. }   \\
\hline   
 2&    0.650E-02 &  3.5&    0.325E+01 &  1.6\\
 3&    0.423E-03 &  3.9&    0.858E+00 &  1.9\\
 4&    0.265E-04 &  4.0&    0.216E+00 &  2.0\\
 5&    0.164E-05 &  4.0&    0.541E-01 &  2.0\\
\hline 
    \end{tabular}%
\end{table}%

\begin{table}[H]
  \centering  \renewcommand{\arraystretch}{1.1}
  \caption{Error profile by the $P_4$ WG element  for computing \eqref{s2}. }
  \label{t3}
\begin{tabular}{c|cc|cc}
\hline
Grid $G_i$ & \quad $\| u-u_h\|_{0}$ & $O(h^r)$ & \  $\|D^2_w(u-u_h)\|_{0}$& $O(h^r)$   \\ \hline
    &  \multicolumn{4}{c}{ On the triangular meshes shown in Figure \ref{f-2}. }   \\
\hline   
 2&    0.752E-04 &  4.9&    0.969E-01 &  2.7\\
 3&    0.241E-05 &  5.0&    0.103E-01 &  3.2\\
 4&    0.828E-07 &  4.9&    0.125E-02 &  3.0\\
 5&    0.258E-08 &  5.0&    0.153E-03 &  3.0\\
\hline 
    &  \multicolumn{4}{c}{ On the   polygonal meshes shown in Figure \ref{f-5}. }   \\
\hline    
 2&    0.110E-02 &  5.0&    0.102E+01 &  2.9\\
 3&    0.342E-04 &  5.0&    0.129E+00 &  3.0\\
 4&    0.106E-05 &  5.0&    0.161E-01 &  3.0\\
 5&    0.715E-07 &  ---&    0.202E-02 &  3.0\\
\hline 
    \end{tabular}%
\end{table}%

\begin{table}[H]
  \centering  \renewcommand{\arraystretch}{1.1}
  \caption{Error profile by the $P_5$ WG element  for computing \eqref{s2}. }
  \label{t4}
\begin{tabular}{c|cc|cc}
\hline
Grid $G_i$ & \quad $\| u-u_h\|_{0}$ & $O(h^r)$ & \  $\|D^2_w(u-u_h)\|_{0}$& $O(h^r)$   \\ \hline
    &  \multicolumn{4}{c}{ On the triangular meshes shown in Figure \ref{f-2}. }   \\
\hline   
 1&    0.117E-02 & --- &    0.436E+00 & --- \\
 2&    0.980E-05 &  6.9&    0.158E-01 &  4.8\\
 3&    0.136E-06 &  6.2&    0.905E-03 &  4.1\\
 4&    0.194E-08 &  6.1&    0.538E-04 &  4.1\\
\hline 
    &  \multicolumn{4}{c}{ On the   polygonal meshes shown in Figure \ref{f-5}. }   \\
\hline    
 1&    0.102E-01 & --- &    0.634E+01 & --- \\
 2&    0.155E-03 &  6.0&    0.390E+00 &  4.0\\
 3&    0.241E-05 &  6.0&    0.242E-01 &  4.0\\
 4&    0.877E-06 &  ---&    0.151E-02 &  4.0\\
\hline 
    \end{tabular}%
\end{table}%

The second numerical example is from \cite{smears,pdwg}.
We solve the elliptic equation in non-divergence form \eqref{model_problem}
   on a square domain $\Omega=(-1,1)\times(-1,1)$, where 
\a{ \p{ a_{11} & a_{12} \\ a_{21} & a_{22}} =\frac{16} 9 \p{ 2  &  {xy}/{|xy|} \\   {xy}/{|xy|} & 2}. } 
We choose an $f$ in  \eqref{model_problem} so that the exact solution is
\an{\label{s5}
   u = xy(1-\t{e}^{1-|x|} )(1-\t{e}^{1-|y|} ) .  }  
 We compute the solution \eqref{s5} on the triangular grids shown in Figure \ref{f-2}, and on the non-convex polygonal 
  grids shown in Figure \ref{f-5}, by 
  the weak Galerkin $P_k$-$P_k$-$P_{k-1}^2$/$P_k^{2\times 2}$ finite elements, $k=2,3,4$ and $5$.
The results are listed in Tables \ref{t5}-\ref{t8}, 
    where we can see that the optimal orders of convergence 
  are achieved in all cases.

\begin{table}[H]
  \centering  \renewcommand{\arraystretch}{1.1}
  \caption{Error profile by the $P_2 $ WG element  for computing \eqref{s5}. }
  \label{t5}
\begin{tabular}{c|cc|cc}
\hline
Grid $G_i$ & \quad $\| u-u_h\|_{0}$ & $O(h^r)$ & \  $\|D^2_w(u-u_h)\|_{0}$& $O(h^r)$   \\ \hline
    &  \multicolumn{4}{c}{ On the triangular meshes shown in Figure \ref{f-2}. }   \\
\hline   
 2&    0.523E-01 &  0.0&    0.337E+01 &  0.0\\
 3&    0.127E-01 &  2.0&    0.221E+01 &  0.6\\
 4&    0.157E-02 &  3.0&    0.106E+01 &  1.1\\
 5&    0.420E-03 &  1.9&    0.578E+00 &  0.9\\
\hline 
    &  \multicolumn{4}{c}{ On the   polygonal meshes shown in Figure \ref{f-5}. }   \\
\hline    
 2&    0.143E+00 &  0.0&    0.614E+01 &  0.0\\
 3&    0.326E-01 &  2.1&    0.407E+01 &  0.6\\
 4&    0.419E-02 &  3.0&    0.208E+01 &  1.0\\
 5&    0.850E-03 &  2.3&    0.107E+01 &  1.0\\
\hline 
    \end{tabular}%
\end{table}%

\begin{table}[H]
  \centering  \renewcommand{\arraystretch}{1.1}
  \caption{Error profile by the $P_3$ WG element  for computing \eqref{s5}. }
  \label{t6}
\begin{tabular}{c|cc|cc}
\hline
Grid $G_i$ & \quad $\| u-u_h\|_{0}$ & $O(h^r)$ & \  $\|D^2_w(u-u_h)\|_{0}$& $O(h^r)$   \\ \hline
    &  \multicolumn{4}{c}{ On the triangular meshes shown in Figure \ref{f-2}. }   \\
\hline   
 2&    0.428E-01 &  0.0&    0.346E+01 &  0.0\\
 3&    0.507E-02 &  3.1&    0.107E+01 &  1.7\\
 4&    0.399E-03 &  3.7&    0.260E+00 &  2.0\\
 5&    0.392E-04 &  3.3&    0.597E-01 &  2.1\\
\hline 
    &  \multicolumn{4}{c}{ On the   polygonal meshes shown in Figure \ref{f-5}. }   \\
\hline    
 2&    0.210E+00 &  0.0&    0.241E+02 &  0.0\\
 3&    0.191E-01 &  3.5&    0.690E+01 &  1.8\\
 4&    0.130E-02 &  3.9&    0.178E+01 &  2.0\\
 5&    0.931E-04 &  3.8&    0.448E+00 &  2.0\\
\hline 
    \end{tabular}%
\end{table}%

\begin{table}[H]
  \centering  \renewcommand{\arraystretch}{1.1}
  \caption{Error profile by the $P_4$ WG element  for computing \eqref{s5}. }
  \label{t7}
\begin{tabular}{c|cc|cc}
\hline
Grid $G_i$ & \quad $\| u-u_h\|_{0}$ & $O(h^r)$ & \  $\|D^2_w(u-u_h)\|_{0}$& $O(h^r)$   \\ \hline
    &  \multicolumn{4}{c}{ On the triangular meshes shown in Figure \ref{f-2}. }   \\
\hline   
 2&    0.126E-01 &  0.0&    0.234E+01 &  0.0\\
 3&    0.471E-03 &  4.7&    0.245E+00 &  3.3\\
 4&    0.181E-04 &  4.7&    0.296E-01 &  3.0\\ 
\hline 
    &  \multicolumn{4}{c}{ On the   polygonal meshes shown in Figure \ref{f-5}. }   \\
\hline    
 2&    0.100E+00 &  0.0&    0.356E+02 &  0.0\\
 3&    0.340E-02 &  4.9&    0.495E+01 &  2.8\\
 4&    0.105E-03 &  5.0&    0.631E+00 &  3.0\\
\hline 
    \end{tabular}%
\end{table}%

\begin{table}[H]
  \centering  \renewcommand{\arraystretch}{1.1}
  \caption{Error profile by the $P_5$ WG element  for computing \eqref{s5}. }
  \label{t8}
\begin{tabular}{c|cc|cc}
\hline
Grid $G_i$ & \quad $\| u-u_h\|_{0}$ & $O(h^r)$ & \  $\|D^2_w(u-u_h)\|_{0}$& $O(h^r)$   \\ \hline
    &  \multicolumn{4}{c}{ On the triangular meshes shown in Figure \ref{f-2}. }   \\
\hline   
 2&    0.209E-02 &  0.0&    0.617E+00 &  0.0\\
 3&    0.526E-04 &  5.3&    0.396E-01 &  4.0\\
 4&    0.129E-05 &  5.3&    0.227E-02 &  4.1\\
\hline 
    &  \multicolumn{4}{c}{ On the   polygonal meshes shown in Figure \ref{f-5}. }   \\
\hline    
 2&    0.181E-01 &  0.0&    0.112E+02 &  0.0\\
 3&    0.329E-03 &  5.8&    0.769E+00 &  3.9\\
 4&    0.559E-04 &  ---&    0.494E-01 &  4.0\\
\hline 
    \end{tabular}%
\end{table}%

\end{document}